\documentclass[12pt,reqno]{amsart}

\usepackage{verbatim}
\usepackage{amsfonts}

\newtheorem{thm}{Theorem}
\newtheorem{prop}[thm]{Proposition}
\newtheorem{cor}[thm]{Corollary}
\newtheorem{lem}[thm]{Lemma}
\theoremstyle{definition}

\newtheorem*{ex}{Example}
\newtheorem{citethm}{Theorem}

\providecommand{\RR}{\mathbb{R}}
\providecommand{\CC}{\mathbb{C}}

\providecommand{\NN}{\mathbb{N}}

\renewcommand{\L}{\mathcal{L}}

\providecommand{\eps}{\epsilon}
\DeclareMathOperator{\extr}{ex}

\DeclareMathOperator{\id}{id}

\providecommand{\eps}{\varepsilon}
\providecommand{\rx}{\RR[x]}
\newcommand{\sos}[1]{\Sigma^2(#1)}
\begin{document}

%\title[Operator-theoretic Positivstellens\" atze]{A survey of  \\ operator-theoretic Positivstellens\" atze \\ in the archimedean case}
\title[Operator-theoretic Positivstellens\" atze]{Archimedean operator-theoretic Positivstellens\" atze}

\author{J. Cimpri\v c}

\keywords{real algebraic geometry, operator algebras, moment problems}

\subjclass[2010]{14P, 13J30, 47A56}

\date{November 22nd, 2010; revised January 16th, 2011}

\address{Jaka Cimpri\v c, University of Ljubljana, Faculty of Math. and Phys.,
Dept. of Math., Jadranska 19, SI-1000 Ljubljana, Slovenija. 
E-mail: cimpric@fmf.uni-lj.si. www page: http://www.fmf.uni-lj.si/ $\!\!\sim$cimpric.}

\begin{abstract}
We prove a general archimedean positivstellensatz for hermitian operator-valued polynomials
and show that it implies the multivariate Fejer-Riesz Theorem of Dritschel-Rovnyak and 
positivstellens\" atze of Ambrozie-Vasilescu and Scherer-Hol.
We also obtain several generalizations of these and related results. 
The proof of the main result depends on an extension of the abstract archimedean positivstellensatz 
for $\ast$-algebras that is interesting in its own right.
\end{abstract}

\maketitle

\thispagestyle{empty}

\section{Introduction}

We fix $d \in \NN$ and write $\rx:= \RR[x_1,\ldots,x_d].$ In real algebraic geometry, 
a \textit{positivstellensatz} is a theorem which for given polynomials $p_1,\ldots,p_m \in \rx$ characterizes all 
polynomials $p \in \rx$ which satisfy $p_1(a)\ge 0,\ldots,p_m(a) \ge 0 \Rightarrow p(a)>0$
for every point $a \in \RR^d$. A nice survey of them is \cite{mm}.
The name \textit{archimedean positivstellensatz} is reserved for the following result of
Putinar \cite[Lemma 4.1]{pu}:

\begin{citethm}
\label{putinar}
Let $S=\{p_1,\ldots,p_m\}$ be a finite subset of  $\rx$. If the set 
$M_S := \{c_0+\sum_{i=1}^m c_i p_i \mid c_0,\ldots,c_m$ are
sums of squares of polynomials from $\rx\}$ 
contains an element $g$ such that the set  $\{x \in \RR^d \mid g(x) \ge 0\}$ is compact,
then for every $p \in \rx$ the following are equivalent:
\begin{enumerate}
\item $p(x) > 0$ on $K_S:=\{x \in \RR^d \mid p_1(x)\ge 0,\ldots,p_m(x)\ge 0\}$.
\item There exists an $\eps >0$ such that $p-\eps \in M_S$.
\end{enumerate}
\end{citethm}

An important corollary of Theorem \ref{putinar} is the following theorem of Putinar and Vasilescu \cite[Corollary 4.4]{pv}.
The case $S=\emptyset$ was first done by Reznick \cite[Theorem 3.15]{rez}, see also \cite[Theorem 4.13]{bw2}.

\begin{citethm}
\label{putinar-vasilescu}
Notation as in Theorem \ref{putinar}.
If $p_1,\ldots,p_m$ and $p$ are homogeneous of even degree
and if  $p(x) > 0$ for every nonzero $x \in K_S$,
then there exists $\theta \in \NN$ such that $(x_1^2+\ldots+x_d^2)^\theta p \in M_S$.
\end{citethm}

Another important corollary of Theorem \ref{putinar} 
(take $S=\{1-x_1^2-y_1^2, x_1^2+y_1^2-1, \ldots, 1-x_d^2-y_d^2,x_d^2+y_d^2-1\} \subseteq \RR[x_1,y_1,\ldots,x_d,y_d]$)
is the following multivariate Fejer-Riesz theorem. 

\begin{citethm}
\label{fejer-riesz}
Every element of $\RR[\cos \phi_1, \sin \phi_1,\ldots,\cos \phi_d,\sin \phi_d]$
which is strictly positive for every $\phi_1,\ldots,\phi_d$ is equal to a sum
of squares of elements from 
$\RR[\cos \phi_1, \sin \phi_1,\ldots,\cos \phi_d,\sin \phi_d]$.
\end{citethm}

Note that Theorem \ref{fejer-riesz} implies neither the classical univariate Fejer-Riesz theorem nor its multivariate extension from
\cite{ns} which both work for nonnegative trigonometric polynomials. 

\medskip

Various generalizations of Theorems \ref{putinar}, \ref{putinar-vasilescu} and \ref{fejer-riesz} have been considered.
Theorem \ref{jacobi} extends Theorems \ref{putinar} and \ref{fejer-riesz} from finite to arbitrary sets $S$ 
and from algebras $\rx$ and $\RR[\cos \phi_1, \sin \phi_1,\ldots,\cos \phi_d,\sin \phi_d]$ to arbitrary algebras of the form $\rx/I$.
It also implies that Theorem \ref{putinar-vasilescu} holds for arbitrary $S$.
It is a special case of Jacobi's representation theorem 
and Schm\" udgen's positivstellensatz,
see \cite[5.7.2 and 6.1.4]{mm}.
Generalizations from sums of squares to sums of even powers and from $\RR$ to subfields of $\RR$ will
not be considered here, see \cite{bw2,jac,mm2}.
%Jacobi's original results \cite[Theorems 6, 7]{jac} and \cite{mm2} say much more.
%A related extension of Schm\" udgen's result \cite[Corollary 3]{sch} is \cite[Corollary 4.4]{bw2}.)

\begin{citethm}
\label{jacobi}
Let $R$ be a commutative real algebra and   
$M$ a quadratic module in $R$ (i.e.
$1 \in M \subseteq R$, $M+M \subseteq M$, $r^2M \subseteq M$ for all $r \in R$).
If $M$ is archimedean  (i.e. for every $r \in R$ we have $l  \pm r \in M$ for some real $l>0$)
then for every $p \in R$ the following are equivalent:
\begin{enumerate}
\item $p \in \eps +M$ for some real $\eps>0$,
\item $\phi(p)>0$ for all $\phi \in V_R:=\mathrm{Hom}(R,\RR)$ such that $\phi(M) \ge 0$.
\end{enumerate}
If $R$ is affine then $M$ is archimedean
iff it contains an element $g$ such that the set $\{\phi \in V_R \mid \phi(g) \ge 0\}$ is compact
in the coarsest topology of $V_R$ for which all evaluations $\phi \mapsto \phi(a), a \in R$, are continuous.
\end{citethm}

We are interested in generalizations of this theory from usual to hermitian operator-valued poly\-nomials, 
i.e. from $\rx$ to $\rx \otimes A_h$ where $A$ is some operator algebra with involution. 
Below, we will survey known generalizations of Theorems \ref{putinar}, \ref{putinar-vasilescu} and \ref{fejer-riesz}
%in this direction 
and formulate our main result which is a generalization of Theorem \ref{jacobi}. 
Such results are of interest in control theory.
They fit into the emerging field of noncommutative real algebraic geometry, see \cite{sch2}.

The first result in this direction was the following generalization of Theorem \ref{putinar-vasilescu}
which was proved by Ambrozie and Vasilescu in \cite{av}, see the last part of their Theorem 8.
We say that an element $a$ of a
$C^\ast$-algebra $A$ is \textit{nonnegative} (i.e.  $a \ge 0$) if $a=b^\ast b$ for some $b \in A$
and that it is \textit{strictly positive} (i.e. $a>0$) if $a-\eps \ge 0$ for some real $\eps>0$.

\begin{citethm}
\label{ambrozie-vasilescu}
Let $A$ be a $C^\ast$-algebra and let $p \in \rx \otimes A_h$ and $p_k \in \rx \otimes M_{\nu_k}(\CC)_h$, $k=1,\ldots,m$,
$\nu_k \in \NN$, be homogeneous polynomials of even degree. Assume that
$K_0:=\{t \in S^{d-1} \mid p_1(t) \ge 0,\ldots,p_m(t) \ge 0\}$
is nonempty and $p(t)>0$ for all $t \in K_0$.
Then there are homogeneous polynomials $q_j \in \rx \otimes A$, $q_{jk} \in \rx \otimes M_{\nu_k \times 1}(A)$,
$j \in J$, $J$ finite, 
and an integer $\theta$ such that
$$(x_1^2+\ldots+x_d^2)^\theta p = \sum_{j \in J} \big( q_j^\ast q_j+\sum_{k=1}^m q_{jk}^\ast p_k q_{jk} \big).$$
\end{citethm}

Our interest in this subject stems from the following generalization of Theorem \ref{putinar} 
which is a reformulation of a result of Scherer and Hol. See \cite[Corollary 1]{hs} for the original result
and \cite[Theorem 13]{ks} for the reformulation and extension to infinite $S$.

\begin{citethm}
\label{klep-schweighofer}
For a finite subset $S=\{p_1,\ldots,p_m\}$ of $M_\nu(\rx)_h$, $\nu \in \NN$, write
$K_S:=\{t \in \RR^d \mid p_1(t)\ge 0,\ldots,p_m(t)\ge 0\}$ and
$M_S:=\{\sum_{j \in J} \big( q_j^\ast q_j+\sum_{k=1}^m q_{jk}^\ast p_k q_{jk} \big) \mid
q_j, q_{jk} \in M_\nu(\rx), j \in J, J \mbox{ finite}\}$. If there is $g \in M_S \cap \rx$
such that the set $\{x \in \RR^d \mid g(x)\ge 0\}$ is compact
(i.e. the quadratic module $M_S \cap \rx$ in $\rx$ is archimedean)
%(cf.  Theorem \ref{jacobi})
then for every $p \in M_\nu(\rx)_h$ such that $p(t)>0$ on $K_S$ we have that $p \in M_S$.
\end{citethm}

Finally, we mention an interesting generalization of Theorem \ref{fejer-riesz} which was 
proved by Dritschel and Rovnyak in \cite[Theorem 5.1]{dr}.

\begin{citethm}
\label{dritschel-rovnyak}
Let $A$ be the $\ast$-algebra of all bounded operators on a Hilbert space. If an element
$$p \in \RR[\cos \phi_1, \sin \phi_1,\ldots,\cos \phi_n,\sin \phi_n] \otimes A_h$$
is strictly positive for every $\phi_1,\ldots,\phi_n$ then $p=\sum_{j \in J} q_j^\ast q_j$ 
for some finite $J$ and $q_j \in \RR[\cos \phi_1, \sin \phi_1,\ldots,\cos \phi_n,\sin \phi_n] \otimes A$.
\end{citethm}

The aim of this paper is to prove the following very general operator-theoretic positivstellensatz
and show that it implies generalizations of Theorems \ref{ambrozie-vasilescu},
\ref{klep-schweighofer} and \ref{dritschel-rovnyak}.
(They will be extended from finite to arbitrary $S$, from $C^\ast$-algebras to
algebraically bounded $\ast$-algebras $A$ and from (trigonometric) polynomials to
 affine commutative real algebras. Theorem \ref{klep-schweighofer} will also be extended
from matrices to more general operators.)
%Algebraically bounded $\ast$-algebras are defined in section \ref{avsec}
%and quadratic modules in $\ast$-algebras are defined in section \ref{facsec}.

\begin{citethm}
\label{newthm}
Let $R$ be a commutative real algebra,
%with trivial involution,
$A$ a real or complex $\ast$-algebra  and $M$ a
quadratic module (cf. section \ref{facsec}) in $R \otimes A$.
If $M$ is archimedean then for every $p \in R \otimes A_h$ the following are equivalent:
\begin{enumerate}
\item $p \in \eps +M$ for some real $\eps>0$.
\item For every multiplicative state $\phi$ on $R$, there exists real $\eps_\phi>0$
such that $(\phi \otimes \id_A)(p) \in \eps_\phi+(\phi \otimes \id_A)(M)$.
\end{enumerate}
If $A$ is algebraically bounded (cf. section \ref{avsec}) and the quadratic module
$M \cap R$ in $R$ is archimedean (cf. Theorem \ref{jacobi}) then $M$ is archimedean.
\end{citethm}

One of the main differences between the operator case and the scalar case is that in the operator case an element of $A_h$
that is not $\le 0$ is not necessarily $>0$. We would like to give an algebraic characterization of operator-valued polynomials 
%(i.e. elements of $\rx \otimes A_h$ where $A$ is an algebraically bounded $\ast$-algebra) 
that are not $\le 0$ in every point from a given set. Every theorem of this
type is called a \textit{nichtnegativsemidefinitheitsstellensatz}. We will prove variants of Theorems
\ref{klep-schweighofer} and \ref{dritschel-rovnyak} that fit into this context.

Finally, we use our results and the main theorem 
from \cite{jo} to get a generalization of the existence result for operator-valued moment problems from \cite{av}
to algebraically bounded $\ast$-algebras.

\section{Factorizable states}
\label{facsec}

Associative unital algebras with involution will be called $\ast$-algebras for short.
Let $B$ be a $\ast$-algebra over $F \in \{\RR,\CC\}$ where $F$ always comes with complex conjugation as involution. 
Write $Z(B)$ for the center of $B$ 
and write $B_h=\{b \in B \mid b^\ast=b\}$ for its set of hermitian elements. Note that the set 
$B_h$ is a real vector space; we assume that it is equipped with the \textit{finest locally convex 
topology}, i.e. the coarsest topology such that every convex absorbing set in $B_h$ is a neighbourhood of zero.

Clearly, every linear functional on $B_h$ is continuous with respect to the finest locally convex topology. 
In other words, the algebraic and the topological dual of $B_h$ are the same; we will write $(B_h)'$ for both. 
We assume that $(B_h)'$ is equipped with the weak*-topology, i.e. topology of pointwise convergence.
We say that $\omega \in (B_h)'$ is \textit{factorizable} if $\omega(xy)=\omega(x)\omega(y)$ for every $x \in B_h$ 
and $y \in Z(B)_h$. Clearly, the set of all factorizable linear functionals on $B_h$ is closed in the weak*-topology.

We say that a subset $M$ of $B_h$ is a \textit{quadratic module} if $1 \in M$, 
$M+M \subseteq M$ and $b^\ast M b \subseteq M$ for every $b \in B$. The smallest quadratic module in $B$ is the set $\sos{B}$
which consists of all finite sums of elements $b^\ast b$ with $b \in B$. The largest quadratic module in $B$ is the set $B_h$.
A quadratic module $M$ in $B$ is \textit{proper} if $M \ne B_h$ (or equivalently, if $-1 \not\in M$.)
Proper quadratic modules in $B$ exist iff $-1 \not\in \sos{B}$.
We say that an element $b \in B_h$ is \textit{bounded} w.r.t. a quadratic module $M$ if 
there exists a number $l \in \NN$ such that $l \pm b \in M$. A quadratic module $M$ is \textit{archimedean}
if every element $b \in B_h$ is bounded w.r.t. $M$
(or equivalently, if $1$ is an interior point of $M$.)
%The quadratic module $B_h$ is always archimedean.

For every subset $M$ of $B_h$ write $M^\vee$ for the set of all $f \in (B_h)'$ which satisfy $f(1)=1$ and $f(M) \ge 0$.
The set of all extreme points of $M^\vee$ will be denoted by $\extr M^\vee$. Elements of $M^\vee$ will be called 
$M$-\textit{positive states} and elements of $\extr M^\vee$ \textit{extreme} $M$-positive states.
A $\sos{B}$-positive state will simply be called a \textit{state}.
Suppose now that $M$ is an archimedean quadratic module. 
Applying the Banach-Alaoglu Theorem to $V=(M-1) \cap (1-M)$ which is a neighbourhood of zero, we see that $M^\vee$ is compact.
The Krein-Milman theorem then implies, that $M^\vee$ is equal to the closure of the convex hull of the set $\extr M^\vee$.
We will show later (see Corollary \ref{corgen})
that $M^\vee$ is non-empty iff $M$ is proper.

Recall that a (bounded) \textit{$\ast$-representation} of $B$ is a homomorphism of unital $\ast$-algebras from $B$ to
the algebra of all bounded operators on some Hilbert space $H_\pi$. We say that a $\ast$-representation $\pi$ of $B$ 
is $M$-\textit{positive} for a given subset $M$ of $B_h$ if $\pi(m)$ is positive semidefinite for every $m \in M$.
For every such $\pi$ and every $v \in H_\pi$ of norm $1$, $\omega_{\pi,v}(x):=\langle \pi(x) v,v \rangle$ belongs to $M^\vee$.
Conversely, if $M$ is an archimedean quadratic module, then every $\omega \in M^\vee$ is of this form by the GNS construction.

The equivalence of (1)-(4) in the following result is sometimes referred to as archimedean positivstellensatz for $\ast$-algebras.
It originates from the Vidav-Handelmann theory, cf. \cite[Section 1]{han} and \cite{vidav}.
Our aim is to add assertions (5) and (6) to this equivalence.

\begin{prop}
\label{propgen}
For every archimedean quadratic module $M$ in $B$ and every element $b \in B_h$
the following are equivalent:
\begin{enumerate}
\item $b \in M^\circ$ (the interior w.r.t. the finest locally convex topology),
\item $b \in \eps +M$ for some real $\eps>0$,
\item $\pi(b)$ is strictly positive definite for every $M$-positive $\ast$-representation $\pi$ of $B$,
\item $f(b) >0$ for every $f \in M^\vee$,
\item $f(b) >0$ for every $f \in \overline{\extr M^\vee}$,
\item $f(b) >0$ for every factorizable $f \in M^\vee$.
\end{enumerate}
\end{prop}

\begin{proof} 
(1) implies (2) because the set $M-b$ is absorbing, hence $-1 \in t(M-b)$ for some $t>0$.
Clearly (2) implies (3). (3) implies (4) because the cyclic $\ast$-representation 
that belongs to $f$ by the $GNS$ construction clearly has the property that
$\pi(m)$ is positive semidefinite for every $m \in M$. (4) implies (1) by the separation
theorem for convex sets. The details can be found in \cite[Theorem 12]{c1} or \cite[Proposition 15]{sch2}
or \cite[Proposition 1.4]{cmn}.

If (5) is true then, by the compactness of $\overline{\extr M^\vee}$, there exists $\eps>0$ such that $f(b) \ge \eps$ for every 
$f \in \overline{\extr M^\vee}$, hence (4) is true by the Krein-Milman theorem. Clearly, (4) implies (6).
By Proposition \ref{propex} below and the fact that the set of all factorizable $M$-positive states is closed, (6) implies (5).
\end{proof}

Similarly, we have the following:

\begin{prop}
\label{propgen2}
For every archimedean quadratic module $M$ in $B$ and every element $b \in B_h$
the following are equivalent:
\begin{enumerate}
\item $b \in \overline{M}$ (the closure w.r.t. the finest locally convex topology),
\item $b +\eps \in M$ for every $\eps>0$,
\item $\pi(b)$ is positive semidefinite for every $M$-positive $\ast$-representation $\pi$ of $B$,
\item $f(b) \ge 0$ for every $f \in M^\vee$,
\item $f(b) \ge 0$ for every $f \in \overline{\extr M^\vee}$,
\item $f(b) \ge 0$ for every factorizable $f \in M^\vee$.
\end{enumerate}
\end{prop}

The following proposition which extends \cite[Ch. IV, Lemma 4.11]{tak} was used in the proof of equivalences (4)-(6) in Propositions \ref{propgen} and \ref{propgen2}.
Its proof depends on the equivalence of (2) and (3) in Proposition \ref{propgen2}. 

\begin{prop}
\label{propex}
If $M$ is an archimedean quadratic  module in $B$ then all extreme $M$-positive states are factorizable.
\end{prop}

\begin{proof}
Pick any $\omega \in \extr M^\vee$ and $y \in Z(B)_h$. We claim that $\omega(xy)=\omega(x)\omega(y)$
for every $x \in B_h$. Since $y=\frac14((1+y)^2-(1-y)^2)$ and $(1\pm y)^2  \in M$, we may assume that $y \in M$.
Since $M$ is archimedean, we may also assume that $1-y \in M$. 

\medskip

Claim: If $\omega(y)=0$, then $\omega(y^2)=0$. 
(Equivalently, if $\omega(1-y)=0$, then $\omega((1-y)^2)=0$.)

\medskip

Since $y, 2-y \in M$, it follows that $1-(1-y)^2 = \frac12 \big( y(2-y)^2+(2-y)y^2 \big)\in M$. 
Since $\omega$ is an $M$-positive state, it follows that $\omega((1-y)^2) \le 1$. On the other hand, 
$\omega((1-y)^2)\omega(1^2) \ge \vert \, \omega((1-y)\cdot 1) \vert^2$ by the Cauchy-Schwartz inequality. 
Now, $\omega(y)=0$ implies that $\omega((1-y)^2)=1$, hence $\omega(y^2)=0$.

\medskip

Case 1: If $\omega(y)=0$, then $\omega(xy)=0$ for every $x \in B_h$.
(Equivalently, if $\omega(1-y)=0$, then $\omega(x(1-y))=0$ for every $x \in B_h$.)
Namely, by the Cauchy-Schwartz inequality and the Claim, $\vert \, \omega(xy) \vert^2 \le \omega(x^2)\omega(y^2)=0$.
It follows that $\omega(xy)=\omega(x)\omega(y)$ if $\omega(y)=0$ or $\omega(y)=1$.

\medskip

Case 2 : If $0 < \omega(y) < 1$, then $\omega_1$ and $\omega_2$ defined by
\[
\omega_1(x):= \frac{1}{\omega(y)} \, \omega(xy) \quad \mbox{ and } \quad
\omega_2(x):= \frac{1}{\omega(1-y)} \, \omega(x(1-y))
\]
($x \in B_h$) are $M$-positive states on $B_h$.
Namely, for every $M$-positive $\ast$-represent\-ation $\pi$ of $B$ and every 
$x \in M$, we have that 
$\pi(xy)=\pi(x)\pi(y)$ is a product of two commuting positive 
semidefinite bounded operators, hence a positive semidefinite bounded operator.
By the equivalence of assertions (2) and (3) in Proposition \ref{propgen2}, 
$xy+\eps \in M$ for every $\eps>0$. Since $\omega$ is $M$-positive, 
it follows that $\omega(xy)\ge 0$ as claimed. Similarly, we prove that $\omega_2$
is $M$-positive. Clearly, $\omega=\omega(y)\omega_1+\omega(1-y)\omega_2$. Since $\omega$ is an extreme
point of the set of all $M$-positive states on $B_h$, it follows that
$\omega=\omega_1=\omega_2$. In particular,  $\omega(xy)=\omega(x)\omega(y)$.
\end{proof}

If we apply Proposition \ref{propgen} or \ref{propgen2} to $b=-1$, we get the following
corollary, parts of which were already mentioned above.

\begin{cor}
\label{corgen}
For every archimedean quadratic module $M$ in $B$ the following are equivalent:
\begin{enumerate}
%\item $-1 \not\in \overline{M}$,
%\item $-1 \not\in M^\circ$,
\item $-1 \not\in M$, 
%(i.e. $M$ is proper),
\item there exists an $M$-positive $\ast$-representation  of $B$,
\item there exists an $M$-positive state on $B$,
\item there exists an extreme $M$-positive state on $B$,
\item there exists a factorizable $M$-positive state on $B$.
\end{enumerate}
\end{cor}

The following variant of Proposition \ref{propgen} which follows easily from Corollary \ref{corgen} 
was proved in \cite[Theorem 5]{c2}. We could call it archimedean nichtnegativsemidefinitheitsstellensatz
for $\ast$-algebras.

\begin{prop}
\label{another}
For every archimedean proper quadratic module $M$ on 
a real or complex $\ast$-algebra $B$ and for every $x \in B_h$, 
the following are equivalent:
\begin{enumerate}
\item For every $M$-positive $\ast$-representation $\psi$ of $B$,
$\psi(x)$ is not negative semidefinite (i.e. $\langle \psi(x)v,v \rangle >0$ for some $v \in H_\psi$).
\item There exists $k \in \NN$ and $c_1,\ldots,c_k \in B$ 
such that $\sum_{i=1}^k c_i x c_i^\ast \in 1+M$.
\end{enumerate}
\end{prop}

\section{Theorems \ref{newthm} and \ref{klep-schweighofer}}

The aim of this section is to prove Theorem \ref{newthm} (see Theorem \ref{mainthm}) and show that it implies
a generalization of Theorem \ref{klep-schweighofer} to compact operators.
We also prove a concrete version of Proposition \ref{another}.

\begin{thm}
\label{mainthm}
Let $R$ be a commutative real algebra with trivial involution,
$A$ a $\ast$-algebra over $F \in \{\RR,\CC\}$ and $M$ an archimedean quadratic module in $B:=R \otimes A$. 
For every element $p$ of $B_h=R \otimes A_h$, the following are equivalent:
\begin{enumerate}
\item $p \in \eps +M$ for some real $\eps>0$.
\item For every multiplicative state $\phi$ on $R$, there exists real $\eps_\phi>0$
such that $(\phi \otimes \id_A)(p) \in \eps_\phi+(\phi \otimes \id_A)(M)$.
\end{enumerate}
The following are also equivalent:
\begin{enumerate}
\item[(1')] $p +\eps \in M$ for every real $\eps>0$.
\item[(2')] For every multiplicative state $\phi$ on $R$ and every real $\eps>0$
we have that $(\phi \otimes \id_A)(p) +\eps \in (\phi \otimes \id_A)(M)$.
\end{enumerate}
Moreover, the following are equivalent:
\begin{enumerate}
\item[(1'')] There exist finitely many $c_i \in B$ such that $\sum_i c_i^\ast p c_i \in 1+M$.
\item[(2'')] For every multiplicative state $\phi$ on $R$ there exist finitely many $d_i \in A$ 
such that $\sum_i d_i^\ast (\phi \otimes \id_A)(p) d_i \in 1+(\phi \otimes \id_A)(M)$.
\end{enumerate}
\end{thm}

\begin{proof}
Clearly (1) implies (2). We will prove the converse in several steps.
Note that for every multiplicative state $\phi$ on $R$, the mapping $\phi \otimes \id_A \colon B \to A$ is
a surjective homomorphism of $\ast$-algebras, hence $(\phi \otimes \id_A)(M)$ is 
an archimedean quadratic module in $A$. Replacing
$B$, $M$, $f$ and $p$ in Proposition \ref{propgen} with $A$, $(\phi \otimes \id_A)(M)$, $\sigma$ and $(\phi \otimes \id_A)(p)$,
we see that (2) is equivalent to
\begin{enumerate}
\item[(A)] For every multiplicative state $\phi$ on $R$ and every state $\sigma$ on $A_h$ such that
$\sigma((\phi \otimes \id_A)(M)) \ge 0$ we have that $\sigma((\phi \otimes \id_A)(p)) > 0$.
\end{enumerate}
Note that $(\phi \otimes \sigma)(r \otimes a)=\phi(r)\sigma(a)=\sigma(\phi(r)a)=\sigma((\phi \otimes \id_A)(r \otimes a))$
for every $r \in R$ and $a \in A_h$. It follows that $\phi \otimes \sigma=\sigma \circ (\phi \otimes \id_A)$. Thus, (A)
is equivalent to
\begin{enumerate}
\item[(B)] for every $M$-positive state on $R \otimes A_h$ of the form $\omega=\phi \otimes \sigma$ 
where $\phi$ is multiplicative, we have that $\omega(p) > 0$.
\end{enumerate}
Since $R \otimes 1 \subseteq Z(B)$, every factorizable state $\omega$ 
satisfies $\omega(r \otimes a)=\omega(r \otimes 1)\omega(1 \otimes a)$
and $\omega(rs \otimes 1)=\omega(r \otimes 1)\omega(s \otimes 1)$
for any $r,s \in R$ and $a \in A_h$. Hence $\omega=\phi \otimes \sigma$
where $\phi$ is a multiplicative state on $R$ and $\sigma$ is a state on $A_h$.
Therefore, (B) implies that
\begin{enumerate}
\item[(C)] $\omega(p) > 0$ for every factorizable $\omega \in M^\vee$.
\end{enumerate}
By Proposition \ref{propgen}, (C) is equivalent to (1).

The equivalence of (1') and (2') can be proved in a similar way using Proposition \ref{propgen2}.
It can also be easily deduced from the equivalence of (1) and (2).

Clearly (1'') implies (2''). Conversely, if (1'') is false, then $-1 \not\in N$
where $N:=\{m-\sum c_i^\ast p c_i \mid m\in M, c_i \in B\}$ 
is the smallest quadratic module in $B$ which contains $M$ and $-p$.
%(\phi \otimes \id_A)(M)
By Corollary \ref{corgen}, there exists a factorizable state $\omega \in N^\vee$.
From the proof of (1) $\Leftrightarrow$ (2), we know that $\omega=\phi \otimes \sigma
=\sigma \circ (\phi \otimes \id_A)$ for a multiplicative state $\phi$ on $R$ and a state
$\sigma$ on $A$. Since $\sigma((\phi \otimes \id_A)(N))=\omega(N)\ge 0$, it follows
 that $-1 \not\in (\phi \otimes \id_A)(N)$. Since
$(\phi \otimes \id_A)(N)=\{(\phi \otimes \id_A)(m)-\sum_j d_j^\ast (\phi \otimes \id_A)(p)d_j \mid m \in M,d_j\in A\}$,
we get that (2'') is false.
\end{proof}

For every Hilbert space $H$ we write $B(H)$ for the set of all bounded operators on $H$,
$P(H)=\sos{B(H)}$ for the set of all positive semidefinite operators on $H$ and $K(H)$
for the set of all compact operators on $H$.

\begin{lem}
\label{compact}
Let $H$ be a separable Hilbert space and $M$ a quadratic module in $B(H)$
which is not contained in $P(H)$. Then $\overline{M}$ contains all hermitian compact operators,
i.e. $K(H)_h \subseteq \overline{M}$.
\end{lem}

\begin{proof}
Let $M$ be a quadratic module in $B(H)$ which is not contained in $P(H)$.
Pick an arbitrary operator $L$ in $M \setminus P(H)$ and a vector $v \in H$ such that $\langle v, Lv \rangle < 0$.
Write $P$ for the orthogonal projection of $H$ on the span of $v$. Clearly, 
$P L P = \lambda P$ where $\lambda < 0$, hence $-P \in M$. If $Q$ is an orthogonal projection of rank $1$,
then $Q=U^\ast P U$ for some unitary $U$, hence $-Q \in M$. Since also $Q=Q^\ast Q \in M$, $M$ contains all
hermitian operators of rank $1$. Therefore, $M$ contains all finite rank operators.
Pick any $K \in K(H)_h \cap P(H)$ and note that $\sqrt{K} \in K(H)_h \cap P(H)$ as well.
Clearly, $-K+ \eps \sqrt{K} \in M$  for every $\eps >0$ since it is a sum of an element from
$K(H)_h \cap P(H)$ and a finite rank operator (check the eigenvalues). It follows that $-K \in \overline{M}$.
It is also clear that every element of $K(H)_h$ is a difference of two elements from 
$K(H)_h \cap P(H)$, hence $K(H)_h \subseteq \overline{M}$.
\end{proof}

As the first application of Theorem \ref{mainthm} and Lemma \ref{compact}, 
we prove the following generalization of Theorem  \ref{klep-schweighofer}.
By Lemma \ref{arch} below, Theorem  \ref{klep-schweighofer} corresponds 
to the case $R=\rx$ and $H$ finite-dimensional, i.e. in the
finite-dimensional case we can omit the assumption on $p$.

\begin{thm}
\label{ksthm}
Let $R$ be a commutative real algebra with trivial involution,
$H$a separable Hilbert space, $M$ an archimedean quadratic module in
$R \otimes B(H)$ and $p$ an element of $R \otimes B(H)$. 

If for every multiplicative state $\phi$ on $R$ there exists a real $\eta_\phi>0$ such that
$(\phi \otimes \id_{B(H)})(p) \in \eta_\phi+P(H)+K(H)_h$ 
(e.g. if $p \in \eta+R \otimes K(H)_h$ for some $\eta>0$) 
then the following are equivalent:
\begin{enumerate}
\item $p \in M^\circ$,
\item For every multiplicative state $\phi$ on $R$ such that $(\phi \otimes \id_{B(H)})(M) \subseteq P(H)$,
there exists $\eps_\phi>0$ such that $(\phi \otimes \id_{B(H)})(p) \in \eps_\phi+P(H)$.
\end{enumerate}

If $(\phi \otimes \id_{B(H)})(p) \in P(H)+K(H)_h$ for every multiplicative state $\phi$ on $R$
(e.g. if $p \in R \otimes K(H)_h$) 
then the following are equivalent:
\begin{enumerate}
\item[(1')] $p \in \overline{M}$,
\item[(2')] For every multiplicative state $\phi$ on $R$ such that $(\phi \otimes \id_{B(H)})(M) \subseteq P(H)$,
we have that $(\phi \otimes \id_{B(H)})(p) \in P(H)$.
\end{enumerate}
\end{thm}

\begin{proof}
Suppose that (1) is true, i.e.  $p \in \eps+M$ for some $\eps>0$.
For every multiplicative state $\phi$ on $R$ such that $(\phi \otimes \id_{B(H)})(M) \subseteq P(H)$
we have that $(\phi \otimes \id_{B(H)})(p) \in (\phi \otimes \id_{B(H)})(\eps+M) \subseteq \eps + P(H)$,
hence (2) is true.
Conversely, suppose that (2) is true. We claim that for every multiplicative state $\phi$ on $R$
there exists $\eps_\phi>0$ such that $(\phi \otimes \id_{B(H)})(p) \in \eps_\phi+(\phi \otimes \id_{B(H)})(M)$. 
Then it follows by Theorem \ref{mainthm} that (1) is true. If $(\phi \otimes \id_{B(H)})(M) \subseteq P(H)$, then 
$(\phi \otimes \id_{B(H)})(p) \in \eps_\phi+P(H) \subseteq \eps_\phi+(\phi \otimes \id_{B(H)})(M)$
by (2) and the fact that $(\phi \otimes \id_{B(H)})(M)$ is a quadratic module in $B(H)$.
On the other hand, if $(\phi \otimes \id_{B(H)})(M) \not\subseteq P(H)$, then 
$K(H)_h \subseteq \overline{(\phi \otimes \id_{B(H)})(M)}$ by Lemma \ref{compact}.
The assumption $(\phi \otimes \id_{B(H)})(p) \in \eta_\phi+P(H)+K(H)_h$ for some $\eta_\phi>0$
then implies that $(\phi \otimes \id_{B(H)})(p) \in \frac{\eta_\phi}{2}+(\phi \otimes \id_{B(H)})(M)$ as claimed. 
The proof of the equivalence $(1') \Leftrightarrow (2')$ is similar.
\end{proof}

In the infinite-dimensional case, the assumption on $p$ cannot be omitted as the following example shows:
\begin{ex}
Let $H$ be an infinite-dimensional separable Hilbert space, $0 \ne E \in B(H)_h$ an orthogonal projection of finite rank
and $T$ an element of $B(H)_h$ such that $T \not\in P(H)+K(H)_h$. Since the quadratic module $\sos{B(H)/K(H)}$ is closed,
also $P(H)+K(H)_h$ is closed, hence there exists a real $\eps>0$ such that $T+\eps   \not\in P(H)+K(H)_h$.
Write $p_1=-x^2 E$, $p_2=1-x^2$ and $p=\eps +x^2 T$. Let $M$ be the quadratic module in $\RR[x] \otimes B(H)$ 
generated by $p_1$ and $p_2$. Since $p_2 \in M$, it follows from Lemma \ref{arch} below that $M$ is archimedean. 
For every point $a \in \RR$ such that $p_1(a) \ge 0$ and $p_2(a) \ge 0$ we have that $a=0$, hence $p(a)=\eps$. 
Therefore, assertion (2) of Theorem \ref{ksthm} is true for our $M$ and $p$. Assertion (1), however,
fails for our $M$ and $p$. If it was true then there would exist finitely many 
$q_i,u_j,v_k \in \RR[x] \otimes B(H)$ and a real $\eta>0$ such that
$p=\eta +\sum_i q_i^\ast q_i+\sum_j u_j^\ast p_1 u_j+\sum_k v_k^\ast p_2 v_k$. For $x=1$, we get
$\eps +T=\eta +\sum_i q_i(1)^\ast q_i(1)-\sum_j u_j(1)^\ast E u_j(1)$. The first two terms belong to $P(H)$
and the last term belongs to $K(H)_h$, a contradiction with the choice of $T$.
\end{ex}

We finish this section with a concrete version of Proposition \ref{another} in the spirit of Theorem \ref{klep-schweighofer}.
For $R=\rx$, we get \cite[Corollary 22]{ks}. 

\begin{thm}
\label{another1}
Let $R$ be a commutative real algebra with trivial involution, $\nu \in \NN$, and
$M$ an archimedean quadratic module in $M_\nu(R)$. For every element $p \in M_\nu(R)_h$,
the following are equivalent:
\begin{enumerate}
\item There are finitely many $c_i \in M_\nu(R)$ such that $\sum_i c_i^\ast p c_i \in 1+M$.
\item For every multiplicative state $\phi$ on $R$ such that $(\phi \otimes \id)(m)$ is
positive semidefinite for all $m \in M$, we have that the operator $(\phi \otimes \id)(p)$ 
is not negative semidefinite.
%, i.e. it has at least one strictly positive eigenvalue.
\end{enumerate}
\end{thm}

\begin{proof} 
Write $A=M_\nu(\RR)$. Clearly, a matrix $C \in A_h$ is not negative semidefinite (i.e. it has at least one strictly positive eigenvalue) iff
there exist matrices $D_i \in A$ such that $\sum_i D_i^\ast C D_i -I$ is positive semidefinite. It follows that a quadratic module
$M$ in $A$ which is different from $\sos{A}$ contains $-I$, hence it is equal to $A_h$. (This also follows from Lemma \ref{compact}.) 
Now we use equivalence (1'') $\Leftrightarrow$ (2'') of Theorem \ref{mainthm}.
\end{proof}

\begin{comment}

For finite-dimensional $H$ we get Klep-Schweighofer:

\begin{cor}
\label{kscor}
Let $\nu$ be an integer, $M$ an archimedean quadratic module in
$M_\nu(R)=R \otimes M_\nu(F)$ and $p$ an
element of $M_\nu(R)_h$. 

The following are equivalent:
\begin{enumerate}
\item $p \in \eps+M$ for some $\eps>0$,
\item For every multiplicative state $\phi$ on $R$ such that $(\phi \otimes \id_{M_\nu(F)})(m)$
is positive semidefinite for every $m \in M$, we have that $(\phi \otimes \id_{M_\nu(F)})(p)$ is strictly positive definite.
\end{enumerate}

Moreover, the following are equivalent:
\begin{enumerate}
\item[(1')] $p + \eps \in M$ for every $\eps>0$,
\item[(2')] For every multiplicative state $\phi$ on $R$ such that $(\phi \otimes \id_{M_\nu(F)})(m)$
is positive semidefinite for every $m \in M$, we have that $(\phi \otimes \id_{M_\nu(F)})(p)$ is positive semidefinite.
\end{enumerate}
\end{cor}

Note that $(\phi \otimes \id_{M_\nu(F)})(p)=[\phi(p_{ij})]_{i,j=1}^n$ if $p=[p_{ij}]_{i,j=1}^n$.

\end{comment}

\section{Theorem \ref{ambrozie-vasilescu}}
\label{avsec}

Recall that a $\ast$-algebra $A$ is \textit{algebraically bounded} if the quadratic module $\sos{A}$ is archimedean. 
For an element $a \in A_h$ we say that $a \ge 0$ iff  $a+\eps \in \sos{A}$
for all real $\eps >0$ (i.e. iff $a \in \overline{\sos{A}}$) and that $a>0$ iff $a \in \eps+ \sos{A}$ for some
real $\eps>0$ (i.e. iff $a \in \sos{A}^\circ$).
It is well-known that every Banach $\ast$-algebra is algebraically bounded.

The aim of this section is to deduce the following theorem from Theorem \ref{mainthm} and to show that
it implies Theorem \ref{ambrozie-vasilescu}. Other applications of Theorem \ref{mainthm} will be discussed in section \ref{finsec}.

\begin{thm}
\label{avthm}
Let $R$ be a commutative real algebra with trivial involution and
$A$ an algebraically bounded $\ast$-algebra over $F \in \{\RR,\CC\}$.  
Let $U$ be an inner product space over $F$, $\L(U)$ the $\ast$-algebra of all adjointable linear operators on $U$,
$\L(U)_+$ its subset of positive semidefinite operators, and $M$ 
an archimedean quadratic module in $R \otimes \L(U)$. 

Write $B:=R \otimes A$ and consider the vector space $B \otimes U$
as left $R \otimes \L(U)$ right $B$ bimodule which is equipped with the biadditive form $\langle \cdot , \cdot \rangle$ defined by
$\langle b_1 \otimes u_1, b_2 \otimes u_2 \rangle := b_1^\ast b_2 \langle u_1,u_2 \rangle_U$.
Write $M'$ for the subset of $B_h$ which consists of all finite sums of elements of the form 
$\langle q, mq \rangle$ where $m \in M$ and $q \in B \otimes U$.

We claim that the set $M'$ is an archimedean quadratic module and that for every element 
$p \in R \otimes A_h$ the following are equivalent:
\begin{enumerate}
\item $p \in \eps+M'$ for some real $\eps>0$.
\item For every multiplicative state $\phi$ on $R$ such that $(\phi \otimes \id_{\L(U)})(M) \subseteq \L(U)_+$,
we have that $(\phi \otimes \id_A)(p) >0$.
\end{enumerate}
Moreover, the following are equivalent:
\begin{enumerate}
\item[(1')] $p \in \overline{M'}$.
\item[(2')] For every multiplicative state $\phi$ on $R$ such that $(\phi \otimes \id_{\L(U)})(M) \subseteq \L(U)_+$,
we have that $(\phi \otimes \id_A)(p) \ge 0$.
\end{enumerate}
Finally, the following are equivalent:
\begin{enumerate}
\item[(1'')] There exist finitely many $c_i \in B$ such that $\sum_i c_i^\ast p c_i \in 1+M'$.
\item[(2'')] For every multiplicative state $\phi$ on $R$ such that $(\phi \otimes \id_{\L(U)})(M) \subseteq \L(U)_+$,
there exist finitely many elements $d_i \in A$ such that $\sum_i d_i^\ast (\phi \otimes \id_A)(p) d_i-1 \ge 0$.
\end{enumerate}
\end{thm}

We will need the following observation which follows from the fact that the set of bounded elements
w.r.t. a given quadratic module is closed for addition and multiplication of commuting elements.

\begin{lem}
\label{arch}
Let $R$ be a commutative algebra with trivial involution and $A$ an algebraically bounded $\ast$-algebra.
A quadratic module $N$ in $R \otimes A$ is archimedean if and only if $N \cap R$ is archimedean in $R$.
If $x_1,\ldots,x_d$ are generators of $R$, then $N$ is archimedean if and only if it contains
$K^2-x_1^2-\ldots-x_d^2$ for some real $K$.
\end{lem}

\begin{proof}[Proof of Theorem \ref{avthm}]
To prove that (1) implies (2), it suffices to prove:

\medskip
\textit{Claim 1.} For every multiplicative state $\phi$ on $R$
such that $(\phi \otimes \id_{\L(U)}) (M) \subseteq \L(U)_+$, we have that
$(\phi \otimes \id_A)(M') \subseteq \sos{A}$.
\medskip

For every $q \in B \otimes U$ and $m \in M$ we have that
$(\phi \otimes \id_A)(\langle q,mq \rangle)=\langle s,(\phi \otimes \id_{\L(U)}) (m) s \rangle$
where $s=(\phi \otimes \id_A \otimes \id_U)(q) \in A \otimes U$.
If $s=\sum_{i=1}^k a_i \otimes u_i$, then
$$\langle s,(\phi \otimes \id_{\L(U)}) (m) s \rangle =
\left[ \begin{array}{ccc} 
a_1^\ast & \ldots & a_k^\ast
\end{array} \right]
T
\left[ \begin{array}{c} 
a_1 \\ \vdots \\ a_k
\end{array} \right]
$$ 
where $T=[\langle u_i,(\phi \otimes \id_{\L(U)}) (m) u_j \rangle ]_{i,j=1}^k \in M_k(F)$.
Since $(\phi \otimes \id_{\L(U)}) (m)$ is positive semidefinite for every $m \in M$,
$T$ is also positive semidefinite.

\medskip

To prove that (2) implies (1), consider the following statement:
\begin{enumerate}
\item[(3)] \textit{For every multiplicative state $\phi$ on $R$ there exists a real $\eps_\phi>0$ such that
$(\phi \otimes \id_A)(p) \in \eps_\phi+(\phi \otimes \id_A)(M')$.}
\end{enumerate}
We claim that (2) implies (3) and (3) implies (1).

Clearly, $b^\ast \langle q,mq \rangle b=\langle qb, mqb \rangle \in M'$ for every $m \in M$, $q \in B \otimes U$ and $b \in B$, 
hence the set $M'$ is a quadratic module in $B$. Clearly, $M' \cap R$ is archimedean since it contains $M \cap R$.
By Lemma \ref{arch}, $M'$ is also archimedean. Hence (3) implies (1) by Theorem \ref{mainthm}.

Suppose that (2) is true and pick a multiplicative state $\phi$ on $R$. 
Clearly, $(\phi \otimes \id_{\L(U)})(M)$ is a quadratic module in $\L(U)$
and $(\phi \otimes \id_A)(M')$ is a quadratic module in $A$.
If $(\phi \otimes \id_{\L(U)})(M) \subseteq \L(U)_+$, then
(2) implies that $(\phi \otimes \id_A)(p) \in \eps_\phi+\sos{A} \subseteq \eps_\phi+(\phi \otimes \id_A)(M')$
for some real $\eps_\phi>0$, hence (3) is true. If $(\phi \otimes \id_{\L(U)})(M) \not\subseteq \L(U)_+$
then (3) follows from:

\medskip
\textit{Claim 2.}   For every multiplicative state $\phi$ on $R$
such that $(\phi \otimes \id_{\L(U)})(M) \not\subseteq \L(U)_+$ 
we have that $(\phi \otimes \id_A)(M')=A_h$.
\medskip

We could use Lemma \ref{compact} but we prefer to prove this claim from scratch.
Pick any $C \in (\phi \otimes \id_{\L(U)})(M) \setminus \L(U)_+$.
There exists $u \in U$ of length $1$ such that $\langle u,Cu \rangle <0$.
Write $P$ for the orthogonal projection of $U$ on the span of $\{u\}$.
Clearly, $P^\ast CP=-\lambda P$ for some $\lambda >0$, hence $-P \in (\phi \otimes \id_{\L(U)})(M)$.
Also, $P=P^\ast P \in (\phi \otimes \id_{\L(U)})(M)$.
Let $m_\pm \in M$ be such that $(\phi \otimes \id_{\L(U)}) (m_\pm)=\pm P$.
Pick any $a \in A_h$ and write $q_\pm = 1_R \otimes \frac{1 \pm a}{2} \otimes u$ where $1=1_A$.
The element $m' = \langle q_+, m_+ q_+ \rangle+\langle q_-, m_- q_- \rangle$
belongs to $M'$ and, by the proof of Claim 1,
$(\phi \otimes \id_A)(m')=\langle s_+, (\phi \otimes \id_{\L(U)})(m_+)s_+ \rangle
+ \langle s_-, (\phi \otimes \id_{\L(U)})(m_-)s_- \rangle$
where $s_\pm = (\phi \otimes \id_A \otimes \id_U)(q_\pm)=\frac{1 \pm a}{2} \otimes u$.
Therefore, $(\phi \otimes \id_A)(m')=
(\frac{1 + a}{2})^2 \langle u,Pu \rangle+(\frac{1 - a}{2})^2 \langle u,-Pu \rangle=
(\frac{1 + a}{2})^2-(\frac{1 - a}{2})^2=a$.

Claim 1 also gives implications (1') $\Rightarrow$ (2') and (1'') $\Rightarrow$ (2'') and Claim 2 
gives their converses. Note that assertion (3) must be replaced with suitable assertions (3') and (3'') 
to which Theorem \ref{mainthm} can be applied.
\end{proof}

\medskip

For $U=F^\nu$, we have that $B \otimes U \cong B^\nu \cong M_{\nu \times 1}(B)$, $R \otimes \L(U) \cong M_\nu(R) \subseteq M_\nu(B)$
and $\langle q, mq \rangle =q^\ast m q$ in Theorem \ref{avthm}. However, if also 
$A=M_\nu(F)$ (i.e. $B=M_\nu(R)$), we do not get Theorem \ref{klep-schweighofer}.
Combining both theorems, we get that archimedean quadratic modules $M$ and $M'$ in 
$M_{\nu}(R)$ have the same interior and the same closure.

Finally, we would like to show that Theorem \ref{avthm} implies Theorem \ref{ambrozie-vasilescu}.
The proof also works for algebraically bounded $\ast$-algebras.

\begin{comment}

\begin{cor}
Let $R=\RR[x_1,\ldots,x_n]$ and $A$ be an algebraically bounded $\ast$-algebra over $F$.
Let $p \in R\otimes A_h$ and $p_k \in M_\nu(R)_h$, $k=1,\ldots,m$,
be homogeneous of even degree. Suppose that for every nonzero point $a \in \RR^n$ such that 
$p_k(a)$, $k=1,\ldots,m$, are positive semidefinite there exists 
$\eps_a>0$ such that $p(a) \in \eps_a +\sos{A}$. Then there exist
$q_j \in R \otimes A_h$ and $q_{jk} \in R \otimes M_{\nu,1}(A)$ ($j \in J$, $J$ finite) 
and a positive integer $\theta$ such that
\[
\Vert x\Vert^{2 \theta} p(x) = \sum_{j \in J} \big( q_j^\ast(x) q_j(x)+\sum_{k=1}^m q_{jk}^\ast(x) p_k(x) q_{jk}(x) \big).
\]
\end{cor}

\end{comment}

\begin{proof}[Proof of Theorem \ref{ambrozie-vasilescu}] Write $\nu=2+\nu_1+\ldots+\nu_m$, $\Vert x \Vert =
\sqrt{x_1^2+\ldots+x_d^2}$ and
$p_0=[1-\Vert x \Vert^2] \oplus [\Vert x \Vert^2-1] \oplus p_1 \oplus \ldots \oplus p_m \in M_\nu(\rx).$
Clearly, $K_0=\{t \in \RR^d \mid p_0(t) \ge 0\}$. Let $M_0$ be the quadratic module in $M_\nu(\rx)$ 
generated by $p_0$. Since $M_0$ contains $(1-\Vert x \Vert^2)I_\nu$, it is archimedean by Lemma
\ref{arch}. By Theorem \ref{avthm} applied to $U=F^\nu$, every element $p \in \rx \otimes A$ which is
strictly positive definite on $K_0$ belongs to $M_0'$. From the definition of $M_0'$, we get that 
$p  =  \sum_{j \in J} \big( s_j^\ast s_j+ q_j^\ast p_0 q_j \big)$ for a finite $J$, 
$s_j \in \rx \otimes A$ and $q_j \in M_{\nu \times 1}(\rx \otimes A)$, hence
\begin{eqnarray*} 
p & = & \sum_{j \in J} \big( s_j^\ast s_j+ w_j^\ast (1-\Vert x \Vert^2) w_j 
+ z_j^\ast (\Vert x \Vert^2-1) z_j+\sum_{k=1}^{m} s_{jk}^\ast p_k s_{jk} \big)
\end{eqnarray*}
for a finite $J$, $s_j,w_j,z_j \in \rx \otimes A$ and $s_{jk} \in M_{\nu_k \times 1}(\rx \otimes A)$. 
Replacing $x$ by $\frac{x}{\Vert x\Vert}$ and multiplying with a large power
of $\Vert x\Vert^2$ we get that 
\begin{eqnarray*}
\Vert x\Vert^{2 \theta} p(x) = \sum_{j \in J} \big( (u_j(x)+\Vert x \Vert v_j(x))^\ast(u_j(x)+\Vert x \Vert v_j(x))+
\\
+\sum_{k=1}^m (u_{jk}(x)+\Vert x \Vert v_{jk}(x))^\ast p_k(x) (u_{jk}(x)+\Vert x \Vert v_{jk}(x)) \big)
\end{eqnarray*}
where $u_j,v_j \in \in \rx \otimes A$ and $u_{jk},v_{jk} \in M_{\nu_k \times 1}(\rx \otimes A)$ for every $j \in J$.
Finally, we can get rid of the terms containing $\Vert x \Vert$ by replacing $\Vert x \Vert$ with $-\Vert x \Vert$ 
and adding the old and the new equation.
\end{proof}

\section{Theorem \ref{dritschel-rovnyak} and moment problems}
\label{finsec}

Our next result, Theorem \ref{sh}, is a special case of Theorem \ref{avthm} for $U=F$.
The proof can be shortened in this case because both claims become trivial.
For $R = \RR[\cos \phi_1, \sin \phi_1,\ldots,\cos \phi_n,\sin \phi_n]$, $M=\sos{R}$ and $A=B(H)$ 
it implies Theorem \ref{dritschel-rovnyak}. 
For $R=\rx$ and $A$ a finite-dimensional $C^\ast$-algebra it implies
\cite[Theorem 2]{hs}, a step in the original proof of Theorem \ref{klep-schweighofer}.
Both special cases can also be obtained from the original proof of 
Theorem \ref{ambrozie-vasilescu} - namely, Theorem 3 and Lemma 5 from \cite{av} 
imply Theorem \ref{sh} for $R=\rx$ and $A$ a $C^\ast$-algebra.

\begin{comment}

Our next result, Theorem \ref{sh}, is a special case of Theorem \ref{avthm} for $U=F$.
(The proof can be shortened in this case because both claims become trivial.)
For $R = \RR[\cos \phi_1, \sin \phi_1,\ldots,\cos \phi_n,\sin \phi_n]$ and $M=\sos{R}$ it gives
a generalization of  Theorem \ref{dritschel-rovnyak} from $A=B(H)$ to any algebraically bounded $\ast$-algebra $A$.
For $R=\rx$ and $A$ a $C^\ast$-algebra, Theorem \ref{sh} already follows from the original proof of
Theorem \ref{ambrozie-vasilescu}  (combine Theorem 3 and Lemma 5 in \cite{av}). 
The first explicit version of Theorem \ref{sh}, for $R=\rx$ and $A$ a finite-dimensional $C^\ast$-algebra,
is \cite[Theorem 2]{hs} and it is a step in the proof of \cite[Corollary 1]{hs}
(i.e. their version of Theorem \ref{klep-schweighofer}).

\end{comment}

\begin{thm}
\label{sh}
Let $R$ be a commutative real algebra with trivial involution,
$A$ an algebraically bounded $\ast$-algebra over $F \in \{\RR,\CC\}$
and $M$ an archimedean quadratic module in $R$. Write $M'=M \cdot \sos{R \otimes A}$ for the quadratic module in
$R \otimes A$ which consists of all finite sums of elements of the form $m q^\ast q$ with $m \in M$ and $q \in R \otimes A$.
For every element $p \in R\otimes A_h$, the following are equivalent:
\begin{enumerate}
\item $p \in \eps+M'$ for some real $\eps>0$.
\item For every multiplicative state $\phi$ on $R$ such that $\phi(M) \ge 0$,
we have that $(\phi \otimes \id_A)(p) >0$.
\end{enumerate}
\end{thm}

If we combine Theorem \ref{sh} with a suitable version 
of the Riesz representation theorem (namely, Theorem 1 in \cite{jo}) we get the following existence result 
for operator-valued moment problems which extends Theorem 3 and Lemma 5 from \cite{av}.
%The notation is from \cite{av}.

\begin{thm}
Let $A$ be an algebraically bounded $\ast$-algebra,
%and $\Vert \cdot \Vert$ the $C^\ast$-seminorm induced by the archimedean quadratic module $\sos{A}$
$R$ a commutative real algebra and $M$ an archimedean quadratic module on $R$.
For every linear functional $L \colon R \otimes A \to \RR$ such that $L(m q^\ast q) \ge 0$ for every $m \in M$
and $q \in R\otimes A$,
there exists an $A'$-valued nonnegative measure $m$ on $M^\vee$ such that $L(p)=\int_{M^\vee} (dm , p)$ for every $p \in R \otimes A$.
(Note that $p$ defines a function $\phi \mapsto (\phi \otimes \id_A)(p)$ from $M^\vee$ to $A$.)
\end{thm}

\begin{proof}
Recall that the set $M^\vee$ is compact in the weak*-topology. We assume that $A$ is equipped 
with its natural $C^\ast$-seminorm induced by the archimedean quadratic module $\sos{A}$, see \cite[Section 3]{c1}, 
hence it is a locally convex $\ast$-algebra. We will write $\mathcal{C}^+(M^\vee,A):=\mathcal{C}(M^\vee,\overline{\sos{A}})$
for the positive cone of $\mathcal{C}(M^\vee,A)$.
Let $i$ be the mapping from $R \otimes A$ to $\mathcal{C}(M^\vee,A)$ defined by
$i(p)(\phi)=(\phi \otimes \id_A)(p)$ for every $p \in R \otimes A$ and $\phi \in M^\vee$.
%Let $M'=M \cdot \sos{R \otimes A}$ be the quadratic module in
%$R \otimes A$ which consists of all finite sums of elements of the form $m q^\ast q$ with $m \in M$ and $q \in R \otimes A$.
By Theorem \ref{sh}, we have that $\mathcal{C}^+(M^\vee,A) \cap i(R \otimes A)=i(\overline{M'})$
where $M'=M \cdot \sos{R \otimes A}$. 
Note that $L$ is a $\overline{M'}$-positive functional on $R \otimes A$ and that is defines in the natural way
an $i(\overline{M'})$-positive functional $L'$ on $i(R \otimes A)$. By the Riesz  extension theorem for positive functionals, 
$L'$ extends to a $\mathcal{C}^+(M^\vee,A)$-positive functional on $\mathcal{C}(M^\vee,A)$ which has the
required integral representation by Theorem 1 in \cite{jo}. Hence $L$ also has the required integral representation.
\end{proof}

Finally, we would like to prove a nichtnegativsemidefinitheitsstellensatz that corresponds to Theorem \ref{sh}.

\begin{thm}
\label{another2}
Let $H$ be a separable infinite-dimensional complex Hilbert space and $R$ a commutative real algebra with trivial involution.
Let $M$ be an archimedean quadratic module in $R$ and $M'=M \cdot \sos{R \otimes B(H)}$. For every $p \in R \otimes B(H)_h$,
the following are equivalent:
\begin{enumerate}
\item There are finitely many $c_i \in R \otimes B(H)$ such that $\sum_i c_i^\ast p c_i \in 1+M'$.
\item For every multiplicative state $\phi$ on $R$ such that $\phi(M) \ge 0$, 
the operator $(\phi \otimes \id_{B(H)})(p)$ is not the sum of a negative semidefinite and a compact operator.
\end{enumerate}
\end{thm}

Note that for finite-dimensional $H$, (1) is equivalent to the following:
For every multiplicative state $\phi$ on $R$ such that $\phi(M) \ge 0$, 
the operator $(\phi \otimes \id_{B(H)})(p)$ is not negative semidefinite;
cf. Theorem \ref{another1}. 

\begin{proof}
The equivalence (1'') $\Leftrightarrow$ (2'') of Theorem \ref{avthm} (with $U=\CC$ and $A=B(H)$) 
says that our assertion (1) is equivalent to the following:

\medskip
(3) \textit{For every multiplicative state $\phi$ on $R$ such that $\phi(M) \ge 0$,
there exist finitely many operators $D_i \in B(H)$ such that $\sum_i D_i^\ast (\phi \otimes \id_A)(p) D_i \in 1+P(H)$.}
\medskip

Therefore it suffices to prove the following claim:

\medskip
\textit{Claim.} For every operator $C \in B(H)_h$, the following are equivalent:
\begin{enumerate}
\renewcommand{\labelenumi}{(\roman{enumi})}
\item $C$ is not the sum of a negative semidefinite and a compact operator,
\item the positive part of $C$ is not compact,
\item there exists an operator $D$ such that $D^\ast C D=1$,
\item there exist finitely many  $D_i \in B(H)$ such that $\sum_i D_i^\ast C D_i \in 1+P(H)$.
\end{enumerate}
The implications (i) $\Rightarrow$ (ii), (iii) $\Rightarrow$ (iv) and (iv) $\Rightarrow$ (i)
are clear. To prove that (ii) implies (iii) we first note that $C_+ := E_0 C=E_0^\ast C E_0$ where
$E_0$ is the spectral projection belonging to the interval $[0,\infty)$.
Since $C_+$ is not compact, there exists by 
the spectral theorem a real number $\gamma>0$ such that the spectral projection $E_\gamma$
belonging to the interval $[\gamma,\infty)$ has infinite-dimensional range. The operator
$C_\gamma:=E_\gamma C_+=E_\gamma^\ast C E_\gamma$ has decomposition $C_\gamma=\tilde{C}_\gamma \oplus 0$
with respect to $H=E_\gamma H \oplus (1-E_\gamma)H$ where $\tilde{C}_\gamma \ge \gamma$. Write 
$F=(\tilde{C}_\gamma)^{-1/2} \oplus 0$ and note that $(E_\gamma F)^\ast C (E_\gamma F)=1 \oplus 0$.
Since $E_\gamma H$ is infinite-dimensional, it is isometric to $H$. If $G$ is an isometry from $H$ onto $E_\gamma H$
then $D:=E_\gamma F G$ satisfies (iii).
\end{proof}


\begin{thebibliography}{99}
\bibitem{av} C.-G. Ambrozie, F.-H.   Vasilescu,           
Operator-theoretic Positivstellensätze, 
Z. Anal. Anwend. \textbf{22} (2003), No. 2, 299--314.
                              

\bibitem{bw} R. Berr, T. W\" ormann, 
Positive polynomials on compact sets.  
Manuscripta Math.  \textbf{104}  (2001),  no. 2, 135--143.

\bibitem{bw2} R. Berr, T. W\" ormann, 
Positive polynomials and tame preorderings. 
Math. Z. \textbf{236} (2001), no. 4, 813–-840.

\bibitem{c1} J. Cimpri\v c, A representation theorem for archimedean quadratic modules on *-rings. Can. Math. Bull. \textbf{52} (2009), 39--52.

\bibitem{c2} J. Cimpri\v c, Maximal quadratic modules on $\ast$-rings. Algebr. Represent. Theory \textbf{11} (2008), no. 1, 83–-91.

\bibitem{cmn} J. Cimpri\v c, M. Marshall, T. Netzer, Closures of quadratic modules, To appear in Trans. Amer. Math. Soc.

\bibitem{dr} M. A. Dritschel, J. Rovnyak, The operator Fejer-Riesz theorem, arXiv: 0903.3639v1.

\bibitem{han} D. Handelman, 
Rings with involution as partially ordered abelian groups,
Rocky Mountain J. Math.  \textbf{11}  (1981), no. 3, 337--381.

\bibitem{jac} T. Jacobi, 
A representation theorem for certain partially ordered commutative rings.  
Math. Z.  \textbf{237}  (2001),  no. 2, 259--273. 


\bibitem{jo} G. W. Johnson,
The dual of $C(S,F)$,
Math. Ann. \textbf{187} (1970), 1--8.

                              

\bibitem{ks} I. Klep, M. Schweighofer, 
Pure states, positive matrix polynomials and sums of hermitian squares, 
arXiv:0907.2260

\bibitem{kr} J.-L. Krivine,  
Anneaux pr\' eordonn\' es. 
J. Analyse Math.  \textbf{12} (1964) 307--326.

\bibitem{mm} M. Marshall, 
Positive polynomials and sums of squares. 
Mathematical Surveys and Monographs, 146. American Mathematical Society, Providence, RI, 2008. xii+187 pp. 
ISBN: 978-0-8218-4402-1; 0-8218-4402-4

\bibitem{mm2} M. Marshall, 
A general representation theorem for partially ordered commutative rings. 
Math. Z. \textbf{242} (2002), no. 2, 217-–225. 

\bibitem{ns} A. Naftalevich, B. Schreiber, 
Trigonometric polynomials and sums of squares. 
Number theory (New York, 1983–-84), 225–-238, 
Lecture Notes in Math., 1135, Springer, Berlin, 1985.


\bibitem{pu} M. Putinar, 
Positive polynomials on compact semi-algebraic sets. 
Indiana Univ. Math. J. \textbf{42} (1993), no. 3, 969--984. 

\bibitem{pv} M. Putinar,  F.-H. Vasilescu, 
Solving moment problems by dimensional extension. 
Ann. of Math. (2) \textbf{149} (1999), no. 3, 1087–-1107. 

\bibitem{rez} B. Reznick,
Uniform denominators in Hilbert's seventeenth problem. (English)
Math. Z. \textbf{220} (1995), no. 1, 75--97 .

\bibitem{hs} C. W. Scherer, C. W. J. Hol,  
Matrix sum-of-squares relaxations for robust semi-definite programs.  
Math. Program.  \textbf{107}  (2006),  no. 1-2, Ser. B, 189--211.

\bibitem{sch} K. Schm\" udgen, 
The $K$-moment problem for compact semi-algebraic sets.
Math. Ann. \textbf{289} (1991), no. 2, 203--206. 

\bibitem{sch2} K. Schm\" udgen, 
Noncommutative real algebraic geometry---some basic concepts and first ideas.  
Emerging applications of algebraic geometry,  325--350, IMA Vol. Math. Appl., 149, Springer, New York, 2009.

\bibitem{st} G. Stengle,  
Nullstellensatz and a positivstellensatz in semialgebraic geometry.  Math. Ann.  \textbf{207}  (1974), 87--97.

\bibitem{tak} M. Takesaki, Theory of Operator Algebras I.
Springer-Verlag, New York-Heidelberg, 1979. vii+415 pp. ISBN: 0-387-90391-7

\bibitem{vidav} I. Vidav, 
On some $\ast$-regular rings, 
Acad. Serbe Sci. Pubi. Inst. Math. \textbf{13} (1959), 73--80.


\end{thebibliography}
\end{document}